\documentclass[11pt]{article}

\usepackage{pgf,acadpgf}
\usepackage{booktabs}
\usepackage{amssymb}
\usepackage{amsmath}
\usepackage{color}
\usepackage{rotating}
\usepackage{theorem}
\usepackage{tikz-cd}

\numberwithin{equation}{section}

\theorembodyfont{\slshape}

  \newtheorem{THM}{Theorem}[section]
  \newtheorem{LEM}[THM]{Lemma}
  \newtheorem{PROP}[THM]{Proposition}
  \newtheorem{COR}[THM]{Corollary}

\theorembodyfont{\rmfamily}

  \newtheorem{DEF}[THM]{Definition}
  \newtheorem{EX}[THM]{Example}

\newif\ifQEDsign
\newcommand{\QED}{\global\QEDsigntrue\hfill$\square$}

\newenvironment{proof}%
    {\par\noindent\textit{Proof.}\global\QEDsignfalse}%
    {\ifQEDsign\else\QED\fi\par\bigskip\par}

\newcommand{\alex}{\mathrel{<_{\mathit{alex}}}}

\newcommand{\leT}{\mathrel{\le_{T}}}
\newcommand{\ltT}{\mathrel{<_{T}}}
\newcommand{\equivT}{\mathrel{\equiv_{T}}}
\newcommand{\geT}{\mathrel{\ge_{T}}}

\renewcommand{\le}{\leqslant}
\renewcommand{\ge}{\geqslant}

\newcommand{\0}{\varnothing}

\renewcommand{\phi}{\varphi}
\renewcommand{\epsilon}{\varepsilon}

\newcommand{\AAA}{\mathbf{A}}
\newcommand{\BB}{\mathbf{B}}
\newcommand{\CC}{\mathbf{C}}

\newcommand{\DD}{\mathbf{D}}

\newcommand{\JJ}{\mathbf{J}}
\newcommand{\KK}{\mathbf{K}}

\newcommand{\NN}{\mathbb{N}}

\newcommand{\RR}{\mathbb{R}}
\renewcommand{\SS}{\mathbf{S}}

\newcommand{\union}{\cup}

\newcommand{\restr}[2]{\hbox{$#1$}\hbox{$|$}_{#2}}

\newcommand{\Boxed}[1]{\mbox{$#1$}}

\newcommand{\id}{\mathrm{id}}

\newcommand{\ID}{\mathrm{ID}}
\newcommand{\Ob}{\mathrm{Ob}}

\newcommand{\LO}{\mathrm{LO}}

\newcommand{\RSurj}{\mathrm{RSurj}}

\newcommand{\op}{\mathrm{op}}

\newcommand{\Age}{\mathrm{Age}}

\newcommand{\iso}{\mathrm{iso}}

\newcommand{\bbF}{\mathbb{F}}

\newcommand{\calA}{\mathcal{A}}
\newcommand{\calB}{\mathcal{B}}
\newcommand{\calC}{\mathcal{C}}
\newcommand{\calD}{\mathcal{D}}

\newcommand{\calF}{\mathcal{F}}

\newcommand{\Fraisse}{Fra\"\i ss\'e}

\newcommand{\Flim}{\mathrm{Flim}}
\newcommand{\Aut}{\mathrm{Aut}}

\newcommand{\GR}{\mathbf{GR}}

\newcommand{\catR}{\mathbf{Ram}}
\newcommand{\catD}{\mathbf{DRam}}
\newcommand{\catG}{\mathbf{Gra}}
\newcommand{\catH}{\mathbf{H}}
\newcommand{\catP}{\mathbf{Pos}}
\newcommand{\catM}{\mathbf{Met}}
\newcommand{\catRel}{\mathbf{Rel}}
\newcommand{\catVec}{\mathbf{Vec}}
\newcommand{\catIR}{\mathbf{InfRam}}

\title{Tukey reducibility for categories -- In search of the strongest statement in finite Ramsey theory}
\author{%
  Keegan Dasilva Barbosa\\
  Fields Institute for Research in Mathematical Science,\\
  222 College St, Toronto, ON M5T 3J1, Canada,\\
  email: keegan.dasilvabarbosa@mail.utoronto.ca
\and
  Dragan Ma\v sulovi\'c\\
  Department of Mathematics and Informatics\\
  Faculty of Sciences, University of Novi Sad, Serbia\\
  email: dragan.masulovic@dmi.uns.ac.rs%
}

\begin{document}

\maketitle

\begin{abstract}
  Every statement of the Ramsey theory of finite structures corresponds
  to the fact that a particular category has the Ramsey property. We can, then, compare the strength of Ramsey
  statements by comparing the ``Ramsey strength'' of the corresponding categories.
  The main thesis of this paper is that establishing pre-adjunctions between pairs of categories is
  an appropriate way of comparing their ``Ramsey strength''. What comes as a pleasant surprise is that
  pre-adjunctions generalize the Tukey reducibility in the same way categories generalize preorders.
  In this paper we set forth a classification program of statements of finite Ramsey theory based on their
  relationship with respect to this generalized notion of Tukey reducibility for categories.
  After identifying the ``weakest'' Ramsey category,
  we prove that the Finite Dual Ramsey Theorem is as powerful as the full-blown version of the Graham-Rothschild Theorem,
  and conclude the paper with the hypothesis that the Finite Dual Ramsey Theorem is the ``strongest'' of all finite Ramsey statements.

  \medskip


  \medskip

\end{abstract}

\section{Introduction}

There is a general feeling that almost every statement in finite Ramsey theory
follows from the Graham-Rothschild Theorem. For example,
Pr\"omel and Voigt write~\cite{promel-voigt-GRPS}:
\begin{quote}
  ``\ldots\ as it turns out, Ramsey's theorem itself is an immediate consequence of
  the Graham-Rothschild theorem. But the concept of parameter sets does not only glue
  arithmetic progressions and finite sets together. Also, it provides a natural
  framework for seemingly different structures like Boolean lattices, partition lattices,
  hypergraphs and Deuber's $(m,p,c)$-sets, just to mention a few. \emph{So, the Graham-Rothschild
  theorem can be viewed as a starting point of Ramsey Theory.}\footnote{emphasis by D.M.}''
\end{quote}
The Graham-Rothschild Theorem (see Example~\ref{p-a-t.ex.2})
is a family of Ramsey statements indexed by pairs $(A, G)$
where $A$ is a finite alphabet an $G$ is a finite group acting on $A$. The proof was first announced in~\cite{GR} as a main
technical step towards the proof of a conjecture by Rota that an analog of the Finite Ramsey Theorem
holds for finitely dimensional vector spaces over a finite field (we shall refer to this statement as \emph{Rota's Ramsey Conjecture}).
The complete proof was published a year later by Graham, Leeb and Rothschild in~\cite{GLR}, and this was one of the
first applications of category theory in finite Ramsey theory.

Comparing the strength of two mathematical statements is easy: a statement $\alpha$ is stronger than
a statement $\beta$ if $\alpha \Rightarrow \beta$. But how does one show that a statement
(such as the Graham-Rothschild Theorem) is \emph{the strongest} statement in a living body of knowledge
(such as the finite Ramsey theory) whose boundaries are vague (what is the \emph{exact list} of statements
of Ramsey theory?), and new statements are being added on daily basis?
There are well-established and deep mathematical disciplines that deal with metaresults of this kind,
such as reverse mathematics and proof theory. But how does one compare the strength of mathematical statements if
one happily accepts the full force of every-day mathematical practice
where the Axiom of Choice is a \emph{condicio sine qua non}, and careful analysis of formal proofs is not a feasible option?

Our starting point is the observation that almost every statement of the Ramsey theory of finite structures
corresponds to the fact that a particular category has the Ramsey property (see Table~\ref{p-a-t.fig.statements-categories}
for a few examples). We can, then, compare the strength of Ramsey statements by comparing the
``Ramsey strength'' of the corresponding categories.
Moreover, we can describe precisely (see Definition~\ref{p-a-t.def.ramsey-cat-of-fin-structs})
the class of categories in which the Ramsey theory of finite structures
resides, and then ask ``What is the `strongest' category in this class?''

\begin{table}
  \centering
  \begin{tabular}{lc}
    \toprule
    Finite Ramsey statement & Category \\
    \midrule
    Finite Ramsey Theorem & $\catR$ \\
    Graham-Rothschild Theorem & $\GR(A, X, G)$\\
    Finite Dual Ramsey Theorem & $\catD^\op$\\
    Rota's Ramsey Conjecture & $\catVec(\bbF)$\\
    Ne\v set\v ril-R\"odl Theorem & $\catRel(L)$\\
    Ramsey property for: \\
    \quad finite ordered graphs & $\catG$\\
    \quad finite $k$-uniform hypergraphs & $\catH(k)$\\
    \quad finite partial orders with linear extensions & $\catP$\\
    \quad finite ordered $S$-metric spaces & $\catM(S)$\\
    \bottomrule
  \end{tabular}
  \caption{Finite Ramsey statements and the corresponding categories
  (see Examples~\ref{p-a-t.ex.1}--\ref{p-a-t.ex.5})}
  \label{p-a-t.fig.statements-categories}
\end{table}

Proving directly that a class of structures has the Ramsey property is usually a laborious task
based on complex combinatorial constructions. Another way of proving Ramsey results
is to start from a context where the Ramsey property has already been established and try to
transfer the results to the context we are interested in.
This strategy was successfully employed by Pr\"omel and Voigt already in~1981 in~\cite{promel-voigt-GRA} where the Ramsey property for finite
ordered graphs was proved by reducing it to the Graham-Rothschild Theorem. This proof was simplified
and the strategy made more accessible in the book \cite{Promel-Book} published in~2013.
In the same year in his paper~\cite{Solecki-Abstract} Solecki introduces an abstract setting
in which a wide variety of classical Ramsey-type results can be proved, and proposes
the notion of interpretability which enables transferring of the Ramsey property between two such abstract settings.
When it comes to modeling the Ramsey-related phenomena in the language of category theory,
it was shown in \cite{masul-preadj} that pre-adjunctions (see Definition~\ref{opos.def.PA})
transport the Ramsey property. This notion is motivated by the careful analysis of
the version of Pr\"omel and Voigt's 1981 proof presented in \cite[Theorem 12.13]{Promel-Book}
where (although not in the language of category theory) a pre-adjunction
from $\catG$ to $\GR(\{0\}, X, \{e\})$ is constructed.
In a recent paper~\cite{Solecki-FRT-CT} Solecki proposed the notion of modeling
\cite[Section 3.3]{Solecki-FRT-CT} which generalizes both the notion of interpretability and the
notion of pre-adjunctions.

The main thesis of this paper is that establishing a pre-adjunction between a pair of categories is
an appropriate way of comparing their ``Ramsey strength,'' not only because pre-adjunctions are able to
transfer all sorts of Ramsey-related phenomena\footnote{partition relation, Ramsey property, small Ramsey degrees,
see e.g.\ \cite{masul-rpppg}} (Solecki's modeling can do all that as well), but because
pre-adjunctions generalize the Tukey reducibility in the same way categories generalize preorders:
\begin{center}
  \begin{tabular}{ccc}
    preorders & $\leadsto$ & Tukey reducibility\\
    $\downarrow$ & & $\downarrow$\\
    categories & $\leadsto$ & pre-adjunctions
  \end{tabular}
\end{center}
Although Tukey reducibility was introduced in \cite{tukey} with the intention to better understand
intricacies of convergence in topology, it has become a very handy tool 
in many other contexts in which some kind of ordering is imposed on the objects under scrutiny
(see e.g.~\cite{dobrinen-survey-tukey,solecki-survey-tukey}).
Detailed analysis of Tukey reducibility in a class of structures often leads to rough classification results
which are invaluable when there are too many isomorphism classes for a
human-readable classification modulo isomorphism. Rough classification of Ramsey categories
may lead to new insights into the profound nature of this formidable combinatorial phenomenon.

We see all this as yet another benefit of looking at Ramsey theory through the lens of category theory.
Category theory not only helps with proving new Ramsey-type results (see for example~\cite{masul-preadj,masul-big-v-small,masul-drpca}),
but also provides us with both the language and the tools to formulate and
formally reason about metaresults of Ramsey theory. For example, one of the results
of this paper shows that the Finite Dual Ramsey Theorem, the most ascetic rendering of the Graham-Rothschild Theorem,
is as powerful as the full-blown version $\GR(A, X, G)$, the obvious candidate for ``the strongest'' of them all.
The explicit and constructive nature of the reductions we use to prove this are
a clear demonstration that, at a small extra cost, Rota's Ramsey Conjecture could have been proved
directly from the Finite Dual Ramsey Theorem, and this was done in 2022 by Barto\v sov\'a, Lopez-Abad, Lupini
and Mbombo~\cite{Bartosova-LopezAbad-Lupini-Mbombo}.
It turns out that relying on finite alphabets and finite groups acting on them to model a context
using the Graham-Rothschild Theorem is just a convenience, not a necessity.
It is important to note, however, that the research community became aware of the
dual Ramsey phenomena, including the Finite Dual Ramsey Theorem, in the early 1980's, some ten years after the
first proof of the Graham-Rothschild Theorem and and Rota's Ramsey Conjecture with it.

The paper is organized as follows. In Section~\ref{p-a-t.sec.prelim} we recall some basic facts
about Tukey reducibility of directed preorders. We then present fundamental
notions of Ramsey theory in the language of category theory and through a sequence of examples introduce
several concrete categories that we shall use to sharpen our tools.
Section~\ref{p-a-t.sec.pre-adj-tukey} is devoted to showing that
pre-adjunctions generalize the Tukey reducibility in the same way categories generalize preorders.
Section~\ref{p-a-t.sec.weakest} then identifies the bottom element in the Tukey ordering of Ramsey categories.
It comes as no surprise that posets, being the categories where the Ramsey property is trivial,
should be the weakest Ramsey categories. However, this fact provides no insight into the mutual relationship of
``proper'' Ramsey statements. We then show that the category $\catR$ which encodes the Finite Ramsey Theorem
is the weakest amongst the most significant classes of categories of finite structures.
In Section~\ref{p-a-t.sec.GR-all-the-same} we show that all the instances of the Graham-Rothschild Theorem
are of the same strength. We conclude the paper with Section~\ref{p-a-t.sec.conclusion} where several
future research directions are indicated.

\section{Preliminaries}
\label{p-a-t.sec.prelim}

\paragraph{Preorders.}
A \emph{preorder} is a set $A$ together with reflexive and transitive relation $\le$.
If $A$ is a preorder, we say that $a, b \in A$ are \emph{equivalent} if $a \le b$ and $b \le a$. We then write
$a \equiv b$. Note that $\equiv$ is an equivalence relation and that $A / \Boxed\equiv$ becomes a partial order if we
order the classes of $\equiv$ so that $[a]_\equiv \le [b]_\equiv$ if and only if $a \le b$ in~$A$.
We shall say that a preorder $A$ is \emph{essentially finite (resp.\ countable)} if $A / \Boxed\equiv$ is finite (resp.\ countable).
A nonempty set $X \subseteq A$ is \emph{bounded (from above)} if there is a $b \in A$ such that $x \le b$ for all $x \in X$.
We then write $X \le b$.
A preorder $A$ is \emph{directed} if every two-element subset $\{x, y\} \subseteq A$ is bounded.
A nonempty set $X \subseteq A$ is \emph{cofinal (in $A$)} if for every $a \in A$ there is an $x \in X$ such that $a \le x$.

\paragraph{Tukey reducibility.}
Let $A$ and $B$ be directed preorders. A map $f : A \to B$ is a \emph{Tukey map} if it is \emph{unbounded} in the following sense:
for every $X \subseteq A$ which is unbounded in $A$ the image $f(X) = \{f(x) : x \in X\} \subseteq B$ is unbounded in~$B$.
A map $f : A \to B$ is \emph{cofinal} if for every $X \subseteq A$ which is cofinal in $A$
the image $f(X) \subseteq B$ is cofinal in~$B$.

\begin{THM}\cite{tukey,schmidt}\label{p-a-t.thm.tukey-fg}
  Let $A$ and $B$ be directed preorders. If there is a Tukey map $f : A \to B$ then there is a cofinal map $g : B \to A$
  such that for all $a \in A$ and $b \in B$:
  $$
    f(a) \mathrel{\le^B} b \Rightarrow a \mathrel{\le^A} g(b).
  $$
\end{THM}

A directed preorder $A$ is \emph{Tukey reducible} to a directed preorder $B$,
in symbols $A \leT B$, if there is a Tukey map $A \to B$.
We write $A \equivT B$ when $A \leT B$ and $B \leT A$ and say that $A$ and $B$ are \emph{Tukey equivalent}.
It is a well-known fact that every essentially countable directed preorder is Tukey equivalent to~1 or~$\omega$.

The following technical statement will be needed later. A map $f : A \to B$ between two preorders is
\emph{monotone} if $x \le y \Rightarrow f(x) \le f(y)$ for all $x, y \in A$.
For an element $a$ of a preorder $A$ let $\langle a] = \{x \in A : x \le a\}$.

\begin{LEM}\label{p-a-t.lem.monotone-tukey}
  Let $A$ and $B$ be essentially countable directed preorders and assume that $A$ is not bounded.
  If there is a Tukey map $A \to B$ then there is a monotone Tukey map $A \to B$.
\end{LEM}
\begin{proof}
  Let us enumerate $A / \Boxed\equiv$ as $\{[a_0]_\equiv, [a_1]_\equiv, [a_2]_\equiv, \ldots\}$
  and let us define $s_i \in A$ and $S_i \subseteq A$, $i \ge 0$, inductively as follows.
  To start the induction let $s_0 = a_0$ and $S_0 = \langle s_0]$.
  Assume that $s_0, \ldots, s_{n-1}$ and $S_0, \ldots, S_{n-1}$ have been constructed.
  Let $j_n = \min\{i \in \omega : a_i \notin S_0 \cup \ldots \cup S_{n-1}\}$ (note that $j_n$ is always
  well-defined because $A$ is not bounded), let $s_n$ be any upper bound for $s_{n-1}$ and $a_{j_n}$ and let
  $S_n = \langle s_n] \setminus (S_0 \cup \ldots \cup S_{n-1})$. Note that:
  \begin{itemize}
    \item $s_0 < s_1 < s_2 < \ldots$,
    \item $\{S_n : n \in \omega\}$ is a partition of $A$, and
    \item if $x \in S_i$, $y \in S_j$ and $x \le y$ then $i \le j$.
  \end{itemize}
  By the assumption, there is a Tukey map $f : A \to B$. Let us construct $\hat f : A \to B$ inductively as follows.
  Put $b_0 = f(s_0)$ and then define $\hat f$ on $S_0$ so that $\hat f(S_0) = \{b_0\}$.
  With $b_{n-1}$ defined, take $b_n$ to be any upper bound
  of $b_{n-1}$ and $f(s_n)$, and define $\hat f$ on $S_n$ so that $\hat f(S_n) = \{b_n\}$.
  If is now easy to verify that $\hat f$ is a well-defined mapping $A \to B$ which is Tukey and monotone.
\end{proof}

\paragraph{Relational structures.}
A \emph{relational language} is a set $L = \{R_i : i \in I\}$ of \emph{relational symbols} where
each $R_i$ comes with its own \emph{arity} $r_i \in \NN$, $i \in I$.
An \emph{$L$-structure} (or a \emph{relational structure} if making $L$ explicit is not relevant)
is a structure $\calA = (A, R^A_i)_{i \in I}$ where $A$ is a set and $R^A_i \subseteq A^{r_i}$ is a relation on $A$ of
arity $r_i$. If $B \subseteq A$ is a set of elements of~$A$ then $\restr \calA B = (B, R^A_i \cap B^{r_i})_{i \in I}$
is the \emph{substructure of $\calA$ induced by $B$}. A mapping $f : A \to B$ is an embedding
of an $L$-structure $\calA$ into an $L$-structure $\calB$, in symbols $f : \calA \hookrightarrow \calB$,
if the following holds for every $i \in I$ and all $a_1, \ldots, a_{r_i} \in A$:
$$
  (a_1, \ldots, a_{r_i}) \in R^A_i \Longleftrightarrow (f(a_1), \ldots, f(a_{r_i})) \in R^B_i.
$$
By $\calA \hookrightarrow \calB$ we indicate that there is an embedding of $\calA$ into $\calB$.

Let us recall some of the basic facts about \Fraisse\ theory~\cite{Fraisse-ThRel}
and the Kechris-Pestov-Todor\v cevi\'c correspondence~\cite{KPT}.
For a countable relational structure $\calF$, the class of all finite substructures of $\calF$
is called the \emph{age} of $\calF$ and we denote it by~$\Age(\calF)$. A class $\KK$ of finite
relational structures is an \emph{age} if there is countable relational structure $\calF$ such that
$\KK = \Age(\calF)$. A class $\KK$ of finite relational structures is an age if and only if
$\KK$ is closed for isomorphisms,
there are at most countably many pairwise nonisomorphic structures in $\KK$,
$\KK$ has the \emph{hereditary property} (if $\calA \in \KK$ and $\calB \hookrightarrow \calA$ then $\calB \in \KK$),
and $\KK$ is directed (for all $\calA, \calB \in \KK$ there is a $\calC \in \KK$ such that $\calA \hookrightarrow \calC$
and $\calB \hookrightarrow \calC$).

An age $\KK$ is a \emph{\Fraisse\ age} (= \Fraisse\ class = amalgamation class) if $\KK$ satisfies the
\emph{amalgamation property}: for all $\calA, \calB, \calC \in \KK$ and embeddings $f : \calA \hookrightarrow \calB$ and
$g : \calA \hookrightarrow \calC$ there exist $\calD \in \KK$ and embeddings $f' : \calB \hookrightarrow \calD$ and
$g' : \calC \hookrightarrow \calD$ such that $f' \circ f = g' \circ g$.
For every \Fraisse\ age $\KK$ there is a unique (up to isomorphism) countable ultrahomogeneous structure $\calF$
such that $\KK = \Age(\calF)$. We say that $\calF$ is the \emph{\Fraisse\ limit} of $\KK$, denoted $\Flim(\KK)$.
Recall that a structure $\calF$ is \emph{ultrahomogeneous} for every $A \in \Age(\calF)$ and any pair of
embeddings $f, g : \calA \hookrightarrow \calF$ there is a $\phi \in \Aut(\calF)$ such that $\phi \circ f = g$.

If $\KK$ is a Ramsey class of finite relational structures which is directed and closed under
isomorphisms and taking substructures then $\KK$ is a \Fraisse\ age~\cite{Nesetril}.
In that case we say that $\KK$ is a \emph{Ramsey age}. So, every Ramsey age is a \Fraisse\ age.

A topological group $G$ is \emph{extremely amenable}
if every continuous action $G \curvearrowright X$ on a compact Hausdroff space $X$ has a joint fixed point,
that is, there is an $x_0 \in X$ such that $g \cdot x_0 = x_0$ for all $g \in G$.
One of the many deep results of \cite{KPT} is the following statement. Let $\KK$ be a \Fraisse\ age
and $F$ its \Fraisse\ limit. Then $\KK$ has the Ramsey property if and only if $\Aut(F)$ is extremely amenable.

\paragraph{Categories.}
In order to specify a \emph{category} $\CC$ one has to specify
a class of objects $\Ob(\CC)$, a class of morphisms $\hom_\CC(A, B)$ for all $A, B \in \Ob(\CC)$,
the identity morphism $\id_A$ for all $A \in \Ob(\CC)$, and
the composition of morphisms~$\cdot$~so that
$\id_B \cdot f = f = f \cdot \id_A$ for all $f \in \hom_\CC(A, B)$, and
$(f \cdot g) \cdot h = f \cdot (g \cdot h)$ whenever the compositions are defined.
If $f \in \hom_\CC(A, B)$ then we say that $A$ is the \emph{domain} and $B$ the \emph{codomain} of~$f$.
A category $\CC$ is \emph{locally small} if $\hom_\CC(A, B)$ is a set for all $A, B \in \Ob(\CC)$.
Sets of the form $\hom_\CC(A, B)$ are then referred to as \emph{hom-sets}.

\begin{EX}
  Every class of first-order structures can be understood as a locally small category whose morphisms are embeddings
  of first-order structures. This is the intended interpretation whenever a class of first-order structures is treated as a category
  and the morphisms are not specified.
\end{EX}

Write $A \to B$ if $\hom_\CC(A, B) \ne \0$. 
A locally small category $\CC$ is \emph{small} if $\Ob(\CC)$ is a set.
A category $\CC$ is \emph{thin} if $|\hom_\CC(A, B)| \le 1$ for all $A, B \in \Ob(\CC)$.

\begin{EX}
  If $\CC$ is a small thin category then $\to$ is a preorder on $\Ob(\CC)$.
  Conversely, every preorder $A$ can be thought of as a small thin category $\AAA$ where
  $\Ob(\AAA) = A$ and for $a, b \in A$ there is a unique morphism $a \to b$ if and only if $a \le b$.
  Consequently, preorders are exactly small thin categories.
\end{EX}

A category $\CC$ is \emph{directed} if for every $A, B \in \Ob(\CC)$ there is a $C \in \Ob(\CC)$ such that $A \to C$ and $B \to C$,
and it \emph{has amalgamation} if for all $A, B_1, B_2 \in \Ob(\CC)$ and morphisms $f_1 \in \hom_\CC(A, B_1)$, $f_2 \in \hom_\CC(A, B_2)$
there is a $C \in \Ob(\CC)$ and morphisms $g_1 \in \hom_\CC(B_1, C)$, $g_2 \in \hom_\CC(B_2, C)$ such that $g_1 \cdot f_1 = g_2 \cdot f_2$.

A morphism $f$ is: \emph{mono} or \emph{left cancellable} if
$f \cdot g = f \cdot h$ implies $g = h$ whenever the compositions make sense;
\emph{epi} or \emph{right cancellable} if
$g \cdot f = h \cdot f$ implies $g = h$ whenever the compositions make sense; and
\emph{invertible} if there is a morphism $g$ with the appropriate domain and codomain
such that $g \cdot f = \id$ and $f \cdot g = \id$.
By $\iso_\CC(A, B)$ we denote the set of all invertible
morphisms $A \to B$, and we write $A \cong B$ if $\iso_\CC(A, B) \ne \0$. Let $\Aut_\CC(A) = \iso_\CC(A, A)$.
An object $A \in \Ob(\CC)$ is \emph{rigid} if $\Aut_\CC(A) = \{\id_A\}$.

Given a category $\CC$, the \emph{opposite category} $\CC^\op$ is a category constructed from $\CC$ on the same class of objects
by formally reversing arrows and composition. More precisely, for $A, B \in \Ob(\CC) = \Ob(\CC^\op)$ we have that
$\hom_{\CC^\op}(A, B) = \hom_{\CC}(B, A)$, and for $f \in \hom_{\CC^\op}(A, B)$ and $g \in \hom_{\CC^\op}(B, C)$
we have that $g \mathbin{\underset{\CC^\op}{\cdot}} f = f \mathbin{\underset{\CC}{\cdot}} g$.

A category $\DD$ is a \emph{subcategory} of a category $\CC$ if $\Ob(\DD) \subseteq \Ob(\CC)$ and
$\hom_\DD(A, B) \subseteq \hom_\CC(A, B)$ for all $A, B \in \Ob(\DD)$.
A category $\DD$ is a \emph{full subcategory} of a category $\CC$ if $\Ob(\DD) \subseteq \Ob(\CC)$ and
$\hom_\DD(A, B) = \hom_\CC(A, B)$ for all $A, B \in \Ob(\DD)$.
A \emph{skeleton of $\CC$} is a full subcategory $\SS$ of $\CC$ such that every object of $\CC$ is isomorphic
to some object in $\SS$, and no two objects of $\SS$ are isomorphic. In other words,
$\SS$ contains exactly one representative of each isomorphism class of objects in~$\CC$.

A \emph{functor} $F : \CC \to \DD$ from a category $\CC$ to a category $\DD$ maps $\Ob(\CC)$ to
$\Ob(\DD)$ and maps morphisms of $\CC$ to morphisms of $\DD$ so that
$F(f) \in \hom_\DD(F(A), F(B))$ whenever $f \in \hom_\CC(A, B)$, $F(f \cdot g) = F(f) \cdot F(g)$ whenever
$f \cdot g$ is defined, and $F(\id_A) = \id_{F(A)}$.
A functor $F : \CC \to \CC$ such that $F(A) = A$ and $F(f) = f$ for all objects $A$ and morphisms $f$
is called the \emph{identity functor} and denoted by~$\ID_\CC$.
Categories $\CC$ and $\DD$ are \emph{isomorphic}, in symbols $\CC \cong \DD$, if there exist
functors $F : \CC \to \DD$ and $G : \DD \to \CC$ such that $G \circ F = \ID_\CC$ and $F \circ G = \ID_\DD$.

A functor $F : \CC \to \DD$ is \emph{full} if it is surjective on homsets (that is: for
every $g \in \hom_\DD(F(A), F(B))$ there is an $f \in \hom_\CC(A, B)$ with $F(f) = g$),
and \emph{faithful} if it is injective on homsets (that is: $F(f) = F(g)$ implies $f = g$).
A functor $F : \CC \to \DD$ is \emph{isomorphism-dense} if for every $D \in \Ob(\DD)$ there is
a $C \in \Ob(\CC)$ such that $F(C) \cong D$. A functor $F : \CC \to \DD$ is an \emph{equivalence}
if it is full, faithful and isomorphism-dense. Categories $\CC$ and $\DD$ are \emph{equivalent}
if there is an equivalence $F : \CC \to \DD$.

\paragraph{Ramsey theory in the language of category theory.}
Basic notions of Ramsey theory of finite structures generalize to locally small categories straightforwardly.
We write
$
  C \longrightarrow (B)^{A}_k
$
to denote that $A \to B \to C$ in $\CC$ and for every $k$-coloring
$
  \chi : \hom_\CC(A, C) \to k
$
there is a morphism $w \in \hom_\CC(B, C)$ such that $|\chi(w \cdot \hom_\CC(A, B))| = 1$.

A category $\CC$ has the \emph{Ramsey property}\footnote{this notion is sometimes referred to as the \emph{embedding Ramsey property},
to distinguish it from the structural Ramsey property where the coloring is applied to subobjects;
in this paper we focus on the embedding Ramsey property exclusively}
if for every integer $k \in \NN$ and all $A, B \in \Ob(\CC)$
such that $A \to B$ there is a $C \in \Ob(\CC)$ such that $C \longrightarrow (B)^{A}_k$.
A category $\CC$ has the \emph{dual Ramsey property} if $\CC^\op$ has the Ramsey property.

\begin{EX}
  Every thin category has the Ramsey property, and this is trivial.
  In particular, the linear order of nonnegative integers $(\omega, \Boxed\le)$
  understood as a thin category has the Ramsey property. We shall denote this category with $\omega$
  and rely on the context to parse the correct interpretation of the symbol (a set, a linear order, or a small thin category).
\end{EX}

\begin{EX}\label{p-a-t.ex.1}
  Let $\catR$ denote the category whose objects are finite chains (linearly ordered sets) and whose morphisms are
  injective monotone maps between them. The fact that $\catR$ has the Ramsey property is a reformulation
  of the Finite Ramsey Theorem:
  \begin{quote} 
    \textbf{Finite Ramsey Theorem.}
    For all positive integers $k$, $\ell$, $m$ there is a positive integer $n$ such that
    for every $n$-element set $C$ and every $k$-coloring of the set $[C]^\ell$ of all $\ell$-element subsets of $C$
    there is an $m$-element subset $B \subseteq C$ such that $[B]^\ell$ is monochromatic.
  \end{quote}
  For future reference let us also state the Infinite Ramsey Theorem~\cite{Ramsey}:
  \begin{quote} 
    \textbf{Infinite Ramsey Theorem.}
    Let $C$ be a countably infinite set. For all positive integers $k$, $\ell$ and for every $k$-coloring
    of the set $[C]^\ell$ of all $\ell$-element subsets of $C$
    there is an infinite subset $B \subseteq C$ such that $[B]^\ell$ is monochromatic.
  \end{quote}
\end{EX}

\begin{EX}\label{p-a-t.ex.2}
  A word $u$ of length $n \ge 1$ over $A$ can be thought of as
  an element of $A^n$ but also as a mapping $u : \{1, 2, \ldots, n\} \to A$. Then
  $u^{-1}(a)$, $a \in A$, denotes the set of all the positions in $u$ where $a$ appears.
  We usually write such words as $u = a_1 a_2 \ldots a_n$ and call them \emph{$n$-letter words} (over~$A$).

  For a finite set $G$, a \emph{$G$-decorated $n$-letter word over $A$} is an $n$-letter word over $A \times G$.
  Instead of $u = (a_1, g_1) \, (a_2, g_2) \, \ldots \, (a_n, g_n) \in (A \times G)^n$
  we will find it beneficial to write $u = a_1^{g_1} a_2^{g_2} \ldots a_n^{g_n}$. We think of $g_i$ as the
  \emph{exponent} of $a_i$.

  Let $X = \{x_1, x_2, \ldots\}$ be a countably infinite set of variables disjoint from $A$ and let $G$ be a finite group
  with the neutral element~$e$.
  An \emph{$m$-parameter $G$-decorated $n$-letter word over $A$}, with $m, n \in \NN$,
  is a word $w : \{1, 2, \ldots, n\} \to (A \union \{x_1, x_2, \ldots, x_m\}) \times G$ satisfying the following:
  \begin{itemize}
  \item
    if $w(i) = (a, g)$ for some $a \in A$ and $g \in G$ then $g = e$ (only $e$ can appear as an exponent of a letter from $A$);
  \item
    for each $\ell \in \{1, \ldots, m\}$ there is an $i \in \{1, \ldots, n\}$ and a $g \in G$ such that $w(i) = (x_\ell, g)$
    (each of the \emph{parameters} $x_1, \ldots, x_m$ appears at least once in $w$);
  \item
    for each $\ell \in \{1, \ldots, m\}$, if $i = \min(w^{-1}(\{x_\ell\} \times G))$ then $w(i) = (x_\ell, e)$
    (the exponent of the first appearance of $x_\ell$ in $w$ has to be $e$; note that $i$ is the position of the first occurrence of
    $x_\ell$ in $w$);
  \item 
    if $k < \ell$ then $\min(w^{-1}(\{x_k\} \times G)) < \min(w^{-1}(\{x_\ell\} \times G))$ (the first appearance of a parameter
    with a lower index has to precede the first appearance of every parameter with the higher index).
  \end{itemize}
  For example, if $A = \{a, b, c, d\}$ and $G = \{e, g, g^2\}$ then the following is a 3-parameter $G$-decorated 12-word over $A$:
  $$
    c^e \, a^e \, x_1^e \, a^e \, x_1^{g^2} \, x_2^e \, d^e \, x_3^e \, x_2^{g^2} \, x_1^g \, a^e \, x_3^g.
  $$
  We shall usually drop $e$ as the exponent and write the above word as:
  $$
    c \, a \, x_1 \, a \, x_1^{g^2} \, x_2 \, d \, x_3 \, x_2^{g^2} \, x_1^g \, a \, x_3^g.
  $$
  Let $W^n_m(A, G)$ denote the set of all the $m$-parameter $G$-decorated $n$-letter words over $A$.

  Assume, now, that $G$ is a finite group acting on $A$ from the right so that $a^g$ denotes the action of $g \in G$ on $a \in A$.
  Then the substitution of one word for the parameters of the other word can be defined as follows.
  For $u \in W^n_m(A, G)$ and $v = v_1^{g_1} v_2^{g_2} \ldots v_m^{g_m} \in W^m_k(A, G)$ let
  $$
    u \cdot v = u[v_1^{g_1}/x_1, v_2^{g_2}/x_2, \ldots, v_m^{g_m}/x_m] \in W^n_k(A, G)
  $$
  denote the word obtained by replacing each occurrence of $x_i$ in $u$ with $v_i^{g_i}$,
  simultaneously for all $i \in \{1, \ldots, m\}$, and ``performing the exponentiation'' so that:
  \begin{itemize}
    \item $(x_\ell^g)^h$ is replaced with $x_\ell^{g \cdot h}$, and
    \item $a^g$ is replaced by the letter obtained by the action of $g \in G$ on $a \in A$.
  \end{itemize}
  For example, let $A = \{a, b, c, d\}$ and $G = \{e, g, g^2\}$ as above (with $g^3 = e$), and let $G$ act on $A$ so that
  $a^g = b$, $b^g = c$, $c^g = a$ and $d^g = d$.
  If $u = c \, a \, x_1 \, a \, x_1^{g^2} \, x_2 \, d \, x_3 \, x_2^{g^2} \, x_1^g \, a \, x_3^g$ and
  $v = b \, x_1 \, x_1^{g^2}$ then
  $$
  \begin{array}{r@{\,}l@{\,}l@{\,}l@{\,}l@{\,}l@{\,}l@{\,}l@{\,}l@{\,}l@{\,}l@{\,}l@{\,}l@{\,}l@{\,}l@{\,}l}
    u \cdot v
    &= c & a & x_1 & a & x_1^{g^2} & x_2 & d & x_3       & x_2^{g^2} & x_1^g & a & x_3^g     &\cdot & \, b \, x_1 \, x_1^{g^2}\\
    &= c & a & b   & a & b^{g^2}   & x_1 & d & x_1^{g^2} & x_1^{g^2} & b^g   & a & x_1^{g^3} \\
    &= c & a & b   & a & a         & x_1 & d & x_1^{g^2} & x_1^{g^2} & c     & a & x_1.      
  \end{array}
  $$

  Let $X = \{x_1, x_2, x_3, \ldots\}$ be a countable set of variables, $A$ a finite alphabet disjoint from~$X$ and $G$ a finite
  group acting on $A$ from the right. By $\GR(A, X, G)$ we denote the \emph{Graham-Rothschild category} whose objects are positive integers 1, 2, \ldots,
  whose morphisms are given by $\hom(k, n) = W^n_k(A, G)$ if $k \le n$ and $\hom(k, n) = \0$ if $k > n$, where the composition
  is the substitution $\cdot$ described above and the identity morphism $\id_n$ is given by $x_1 x_2 \ldots x_n \in W^n_n(A, G)$.
  The famous Graham-Rothschild Theorem~\cite{GR,GLR} states that every Graham-Rothschild category $\GR(A, X, G)$ has the Ramsey property:

  \begin{quote}
    \textbf{Graham-Rothschild Theorem.}
    For every choice of positive integers $\ell, m, k \ge 1$ there exists an $n \in \NN$ such that
    for every coloring $\chi : W^n_\ell(A, G) \to k$ there exists a $u \in W^n_m(A, G)$ such that
    $|\chi(\{u \cdot v : v \in W^m_\ell(A, G)\})| = 1$.
  \end{quote}
\end{EX}

\begin{EX}\label{p-a-t.ex.3}
  A surjective function $f : A \to B$ between two finite chains $A$ and $B$ is \emph{rigid}
  if $\min f^{-1}(b) < \min f^{-1}(b')$ whenever $b < b'$ in~$B$. For finite chains $A$ and $B$
  let $\RSurj(A, B)$ denote the set of all rigid surjections $A \to B$.
  Let $\catD$ be the category whose objects are all finite chains and morphisms are rigid surjections between them.

  Parameter words from $W_m^n(\0, \{e\})$ are clearly related to rigid surjections.
  To an $m$-parameter $n$-letter word $u = u_1 u_2 \ldots u_n \in W^n_m(\0, \{e\})$ we assign a rigid surjection
  $f_u : \{1 < \ldots < n\} \to \{1 < \ldots < m\}$ so that $f_u(i) = j$ if and only if $u(i) = x_j$.
  It is easy to see that the substitution of parameter words corresponds precisely to
  the composition of rigid surjections, albeit in the opposite direction:
  $$
    f_{u \cdot v} = f_v \circ f_u.
  $$
  This immediately yields that the skeleton of $\catD^\op$ is isomorphic to $\GR(\0, X, \{e\})$,
  so $\catD^\op$ has the Ramsey property. Therefore, $\catD$ has the dual Ramsey property.
  This is a reformulation of the Finite Dual Ramsey Theorem:
  
  \begin{quote}
    \textbf{Finite Dual Ramsey Theorem.}
    For every $k \in \NN$ and finite chains $A$ and $B$ there exists a finite chain $C$ such that
    for every coloring $\chi : \RSurj(C, A) \to k$ there exists a $w \in \RSurj(C, B)$ such that
    $|\chi(\RSurj(B, A) \circ w)| = 1$.
  \end{quote}
\end{EX}

\begin{EX}\label{p-a-t.ex.3a}
  Let us now present the ordered version of Rota's Ramsey Conjecture using
  the ordering of finite vector spaces suggested in \cite{Thomas}.
  Let $\bbF$ be a finite field and $<$ a linear ordering of $\bbF$ such that that $0 < \alpha$ for every
  $\alpha \in \bbF \setminus \{0\}$. A \emph{naturally ordered $n$-dimensional vector space over $\bbF$}
  is a structure $(\bbF^n, \Boxed\alex)$ where $\alex$ is the anti-lexicographic ordering of tuples defined by
  $
    (a_1, a_2, \ldots, a_n) \alex (b_1, b_2, \ldots, b_n)
  $
  if there is an index $i$ such that $a_i < b_i$ and $a_j = b_j$ for all $j > i$.
  The objects of the category $\catVec(\bbF)$ are naturally ordered $n$-dimensional vector space over $\bbF$
  for all $n \in \NN$, and its morphisms are monotone linear maps between them.
  The Ramsey property for $\catVec(\bbF)$ was established in~\cite{KPT}.
\end{EX}

\begin{EX}\label{p-a-t.ex.4}
  An \emph{ordered graph} is a structure $(V, E, \Boxed<)$ where $(V, E)$ is a graph (simple, undirected, no loops)
  and $<$ is an arbitrary linear order on $V$. The category $\catG$ has finite ordered graphs as objects, and embeddings between them
  as morphisms. The Ramsey property for $\catG$ was established several times \cite{Nesetril-Rodl,AH,promel-voigt-GRA}.

  An \emph{ordered $k$-uniform hypergraph} is a structure $(V, E, \Boxed<)$ where $(V, E)$ is a $k$-uniform hypergraph
  ($E$ consists of $k$-element subsets of $V$) and $<$ is an arbitrary linear order on $V$.
  The category $\catH(k)$ has finite ordered $k$-uniform hypergraphs as objects, and embeddings between them
  as morphisms. The Ramsey property for $\catH(k)$ was also established several times \cite{Nesetril-Rodl,AH}.

  A \emph{partial order with a linear extension} is a structure $(A, \Boxed\sqsubseteq, \Boxed<)$ where
  $(A, \Boxed\sqsubseteq)$ is a partially ordered set and $<$ is a linear order on $A$ which extends $\sqsubseteq$
  (that is, if $a \sqsubseteq b$ and $a \ne b$ then $a < b$).
  The category $\catP$ has finite partial orders with linear extensions as objects, and embeddings between them as morphisms.
  The Ramsey property for $\catP$ was established in \cite{Nesetril-Rodl-1984,sokic}.

  For $S \subseteq \RR$, a \emph{finite ordered $S$-metric space} is a structure $(M, d, \Boxed<)$ where $(M, d)$ is a metric space,
  $<$ is a linear order on $M$ and $d(x, y) \in S$ for all $x, y \in M$.
  The category $\catM(S)$ has finite ordered $S$-metric spaces as objects, and monotone isometric embeddings as morphisms.
  For well-behaved distance sets $S$ the Ramsey property for $\catM(S)$ was established in \cite{Nesetril-metric}.
\end{EX}

\begin{EX}\label{p-a-t.ex.5}
  Let $L = \{R_i : i \in I\}$ be a relational language.
  An \emph{ordered $L$-structure} is a structure $(A, \Boxed<, R^A_i)_{i \in I}$
  where $<$ is a linear order on $A$ and $\Boxed< \notin L$. The category $\catRel(L)$ has finite ordered
  $L$-structures as objects, and monotone embeddings as morphisms.
  The Ramsey property for $\catRel(L)$ was established independently in \cite{AH} and \cite{Nesetril-Rodl}.
\end{EX}

\section{Pre-adjunctions generalize Tukey reducibility}
\label{p-a-t.sec.pre-adj-tukey}

In this section we prove that within the class of small thin categories ($=$~preorders)
the existence of a pre-adjunction coincides with Tukey reducibility.
Thus, pre-adjunctions generalize Tukey reducibility in the same way categories generalize preorders.
Let us start by recalling the definition of pre-adjunction from~\cite{masul-preadj}.

\begin{DEF}\label{opos.def.PA} \cite{masul-preadj}
  Let $\BB$ and $\CC$ be locally small categories. A pair of maps
  $
    F : \Ob(\BB) \rightleftarrows \Ob(\CC) : H
  $
  is a \emph{pre-adjunction from $\BB$ to $\CC$} provided there is a family of maps
  $
    \Phi_{X,Y} : \hom_\CC(F(X), Y) \to \hom_\BB(X, H(Y))
  $
  indexed by the pairs $(X, Y) \in \Ob(\BB) \times \Ob(\CC)$ and satisfying the following:
  \begin{itemize}
  \item[(PA)]
  for every $C \in \Ob(\CC)$, every $A, B \in \Ob(\BB)$,
  every $u \in \hom_\CC(F(B), C)$ and every $f \in \hom_\BB(A, B)$ there is a $v \in \hom_\CC(F(A), F(B))$
  satisfying $\Phi_{B, C}(u) \cdot f = \Phi_{A, C}(u \cdot v)$.
  \end{itemize}
  \begin{center}
    \begin{tikzcd}
        B \arrow[rr, "\Phi_{B, C}(u)"]                          &      & H(C)      & & F(B) \arrow[rr, "u"]                          &     & C \\
        A \arrow[u, "f"] \arrow[urr, "\Phi_{A, C}(u \cdot v)"'] &     &        & &  F(A) \arrow[u, "v"] \arrow[urr, "u \cdot v"']                \\
                                                            & \BB \arrow[rrrr, shift left, "F"]
                                                            &       & &                                                                      & \CC \arrow[llll, shift left, "H"]
    \end{tikzcd}
  \end{center}
\end{DEF}
\noindent
Note that in a pre-adjunction $F$ and $H$ are \emph{not} required to be functors, just maps from the class of objects of one of the two
categories into the class of objects of the other category; also $\Phi$ is not required to be a natural isomorphism, just a family of
maps between homsets satisfying the requirement above.

Let us now move on to showing that pre-adjunctions between categories properly generalize Tukey reducibility for preorders.

\begin{LEM}\label{p-a-t.lem.tukey-cofinal}
  Let $\BB$ and $\CC$ be small categories and let
  $
    F : \Ob(\BB) \rightleftarrows \Ob(\CC) : H
  $
  be a pre-adjunction from $\BB$ to $\CC$. Then $F$ is a Tukey map and $H$ is a cofinal map
  if we take $\Ob(\BB)$ and $\Ob(\CC)$ as sets preordered by $\to$.
\end{LEM}
\begin{proof}
  Let $\Phi_{X,Y} : \hom_\CC(F(X), Y) \to \hom_\BB(X, H(Y))$ be the corresponding family of maps for the pre-adjunction.

  Let us first show that $F$ is a Tukey map, that is: if $\{B_i : i \in I\} \subseteq \Ob(\BB)$ is unbounded
  then $\{F(B_i) : i \in I\} \subseteq \Ob(\CC)$ is unbounded. Suppose, to the contrary, that
  $\{F(B_i) : i \in I\}$ is bounded and let $C \in \Ob(\CC)$ be the upper bound for $\{F(B_i) : i \in I\}$.
  Then $F(B_i) \to C$ for all $i \in I$, so for every $i \in I$ there is a morphism $f_i \in \hom(F(B_i), C)$.
  But then $\Phi_{B_i, C}(f_i) \in \hom(B_i, H(C))$ whence follows that $H(C)$ is an upper bound for
  $\{B_i : i \in I\}$. Thus, $\{B_i : i \in I\}$ is bounded.

  Let us now show that $H$ is a cofinal map. Assume that $\{C_i : i \in I\} \subseteq \Ob(\CC)$ is cofinal in $\CC$
  and let us show that $\{H(C_i) : i \in I\} \subseteq \Ob(\BB)$ is cofinal in $\BB$.
  Take any $A \in \Ob(\BB)$. Then $F(A) \in \Ob(\CC)$ so there is an $i_0 \in I$ such that $F(A) \to C_{i_0}$.
  Take any $f \in \hom(F(A), C_{i_0})$. Then $\Phi_{A, C_{i_0}}(f) \in \hom(A, H(C_{i_0}))$, i.e.\ $A \to H(C_{i_0})$.
  This proves that $\{H(C_i) : i \in I\}$ is cofinal in $\BB$.
\end{proof}

The following statement shows that pre-adjunctions properly generalize Tukey reducibility.
Recall that each preorder also has an alter ego in the form of a thin category. To make the proof
more palatable we shall use the following typographic convention: if $B$ is a preorder as a
relational structure, then $\BB$ will denote the same preorder as a thin category.
Note that in this case $\Ob(\BB) = B$.

\begin{THM}\label{p-a-t.thm.tukey=PA}
  Let $B$ and $C$ be essentially countable directed preorders.
  Then $B \leT C$ as preordered sets if and only if there is a pre-adjunction from $\BB$ to $\CC$ understood as thin categories.
\end{THM}
\begin{proof}
  $(\Leftarrow)$ Assume that there is a pre-adjunction $F : \Ob(\BB) \rightleftarrows \Ob(\CC) : H$.
  Then, by Lemma~\ref{p-a-t.lem.tukey-cofinal}, $F$ is a Tukey map from $B$ to $C$,
  so $B \leT C$.

  $(\Rightarrow)$  Let $f : B \to C$ be a Tukey map.
  By Lemma~\ref{p-a-t.lem.monotone-tukey} we may safely assume that $f$ is monotone.
  Theorem~\ref{p-a-t.thm.tukey-fg} then tells us that there is a cofinal map
  $h : C \to B$ such that for all $x \in A$ and $y \in B$:
  $$
    f(x) \mathrel{\le^C} y \Rightarrow x \mathrel{\le^B} h(y).
  $$
  Therefore, if $\hom_\CC(f(x), y) \ne \0$ then $\hom_\BB(x, h(y)) \ne \0$. Since both $\BB$ and $\CC$ are
  thin categories this suffices to conclude that there is a family of maps
  $\Phi_{x, y} : \hom_\CC(f(x), y) \to \hom_\BB(x, h(y))$. Namely, the hom-sets in both $\BB$ and $\CC$ are
  at most one element sets, so what we really needed to ensure is that the codomain of $\Phi_{x,y}$ is nonempty
  whenever its domain is nonempty. The fact that $f$ is monotone ensures that the condition (PA) is satisfied,
  so $f : \Ob(\CC) \rightleftarrows \Ob(\BB) : h$ is indeed a pre-adjunction.
\end{proof}

Note, also, that the relationship ``there is a pre-adjunction from $\BB$ to $\CC$''
behaves as a preorder for locally small categories. Reflexivity is obvious, and transitivity is not much harder.
Namely, if $F : \Ob(\BB) \rightleftarrows \Ob(\CC) : H$
is a pre-adjunction from $\BB$ to $\CC$ together with a corresponding family of maps
  $$
    \Phi_{X,Y} : \hom_\CC(F(X), Y) \to \hom_\BB(X, H(Y)),
  $$
and if $J : \Ob(\CC) \rightleftarrows \Ob(\DD) : K$
is a pre-adjunction from $\CC$ to $\DD$ together with a corresponding family of maps
  $$
    \Psi_{Y,Z} : \hom_\DD(J(Y), Z) \to \hom_\CC(Y, K(Z)).
  $$
then it is easy to check that
  $
    J \circ F : \Ob(\BB) \rightleftarrows \Ob(\DD) : H \circ K
  $
together with the family of maps
  $$
    \Xi_{X, Z} = \Phi_{X,K(Z)} \circ \Psi_{F(X),Z} : \hom_\DD(J\circ F(X), Z) \to \hom_\BB(X, H \circ K(Z))
  $$
is a pre-adjunction from $\BB$ to $\DD$.

All these simple facts motivate the following:

\begin{DEF}
  Let $\BB$ and $\CC$ be locally small categories. We say that $\BB$ is \emph{Tukey reducible} to $\CC$,
  and write $\BB \leT \CC$, if there is a pre-adjunction from $\BB$ to $\CC$.
  If $\BB \leT \CC$ and $\CC \leT \BB$ we say that $\BB$ and $\CC$ are \emph{Tukey equivalent}
  and write $\BB \equivT \CC$. Moreover, we write $\BB \ltT \CC$ if $\BB \leT \CC$ and $\BB \not\equivT \CC$.
\end{DEF}

\begin{EX}
  In~\cite{masul-preadj} the Ramsey property for $\catP$ and $\catM(S)$ for certain well-behaved distance sets $S$
  was proved by showing that $\catM(S) \leT \catP \leT \GR(\{0\}, X, \{e\})$.
\end{EX}

\begin{EX}
  In~\cite{masul-nrt} the Ramsey property for $\catH(k)$ was proved by showing that $\catH(k) \leT \GR(\{0\}, X, \{e\})$.
  For an arbitrary relational language $L$, proving the Ramsey property for $\catRel(L)$ is then just a matter of
  careful bookkeeping.
\end{EX}

\begin{EX}
  In~\cite{Bartosova-LopezAbad-Lupini-Mbombo} the Ramsey property for $\catVec(\bbF)$
  was proved by showing that $\catVec(\bbF) \leT \catD^\op$ for every finite field~$\bbF$.
\end{EX}

Let us conclude the section with a few unsurprising facts which we list here as a technicality,
but also to show that the notion we have introduced conforms to our intuition.

\begin{LEM}\label{p-a-t.lem.surj-ftr}
  Let $\BB$ and $\CC$ be locally small categories such that there is a functor $H : \CC \to \BB$ which is
  full and isomorphism dense. Then there is a map $F : \Ob(\BB) \to \Ob(\CC)$ such that
  $F : \Ob(\BB) \rightleftarrows \Ob(\CC) : H$ is a pre-adjunction.
\end{LEM}
\begin{proof}
  Since $H : \Ob(\CC) \to \Ob(\BB)$ is isomorphism-dense for every $B \in \Ob(\BB)$ choose
  $F(B) \in \Ob(\CC)$ so that $B \cong H(F(B))$ and then choose an isomorphism $\eta_B \in \hom_\BB(B, H(F(B)))$.
  Define $\Phi_{B,C} : \hom_\CC(F(B), C) \to \hom_\BB(B, H(C))$
  by $\Phi_{B,C}(u) = H(u) \cdot \eta_B$. To see that this constitutes a pre-adjunction we still have to verify (PA).
  Take any $A, B \in \Ob(\BB)$, any $C \in \Ob(\CC)$, a morphism $f \in \hom_\BB(A, B)$ and a morphism $u \in \hom_\CC(F(B), C)$.
  Since $H$ is full there is a morphism $f' \in \hom_\CC(F(A), F(B))$ such that
  $H(f') = \eta_B \cdot f \cdot \eta_A^{-1} \in \hom_\BB(H(F(A)), H(F(B)))$. An easy computation now verifies (PA):
  \begin{align*}
    \Phi_{A,C}(u \cdot f')
      &= H(u \cdot f') \cdot \eta_A = H(u) \cdot H(f') \cdot \eta_A\\
      &= H(u) \cdot \eta_B \cdot f \cdot \eta_A^{-1} \cdot \eta_A = \Phi_{B,C}(u) \cdot f.    
  \end{align*}
  This concludes the proof.
\end{proof}

As an immediate corollary we have the following:

\begin{LEM}\label{p-a-t.lem.SeqpaC}
  Let $\BB$ and $\CC$ be locally small categories.

  $(a)$ If $\BB$ and $\CC$ are equivalent categories then $\BB \equivT \CC$.
  
  $(b)$ If $\BB \cong \CC$ then $\BB \equivT \CC$.

  $(c)$ If $\BB$ is a skeleton of $\CC$ then $\BB \equivT \CC$.
\end{LEM}

\begin{COR}\label{p-a-t.cor.FRT-le-FDRT}
  $\catR \leT \catD^\op$.
\end{COR}
\begin{proof}
  Let $\BB$ be the skeleton of $\catR$ spanned by finite chains of the form $n = \{0, 1, \ldots, n-1\} \in \NN$,
  and let $\CC$ be the skeleton of $\catD$ spanned by the same finite chains. For a rigid surjection
  $f : n \to m$ define a monotone map $f^\partial : m \to n$ by
  $
    f^\partial(i) = \min f^{-1}(i)
  $.
  It is a well-known fact that $\mathstrut^\partial$ is functorial, so
  $H : \CC \to \BB$ given by $H(n) = n$ on objects and $H(f) = f^\partial$ on morphisms
  is a functor which is surjective on both objects and homsets.
  Therefore, $\catR \equivT \BB \leT \CC \equivT \catD$ by Lemmas~\ref{p-a-t.lem.surj-ftr}
  and~\ref{p-a-t.lem.SeqpaC}.
\end{proof}

\section{The weakest Ramsey category}
\label{p-a-t.sec.weakest}

As we have just established, $\leT$ is a preordering of locally small categories
with property that $\CC \geT \BB$ implies that $\CC$ is ``Ramsey stronger'' than $\BB$.
In this section we restrict our attention to Ramsey categories of finite objects, which are
the appropriate abstraction of classes of finite relational structures, and in this context
identify the smallest element with respect to~$\leT$. It comes as no surprise that the ``weakest''
Ramsey category is~$\omega$.

\begin{DEF}\label{p-a-t.def.ramsey-cat-of-fin-structs}
  We shall say that a category $\CC$ is a \emph{Ramsey category of finite objects} if:
  \begin{itemize}
    \item $\CC$ is a locally small directed category whose morphisms are mono;
    \item $\CC$ has the Ramsey property;
    \item the skeleton $\SS$ of $\CC$ has at most countably many objects;
    \item for every $S \in \Ob(\SS)$ there are only finitely many morphisms in $\SS$ whose codomain is $S$.
  \end{itemize}
\end{DEF}

\begin{LEM}\label{p-a-t.lem.misc}
  Let $\CC$ be a Ramsey category of finite objects. Then:

  $(a)$ $\hom_\CC(A, B)$ is finite for all $A, B \in \CC$;

  $(b)$ for every $C \in \Ob(\CC)$ there are, up to isomorphism, only finitely many objects $B \in \Ob(\CC)$ such that $B \to C$;

  $(c)$ if $A \cong B$ for some $A, B \in \Ob(\CC)$ then $\hom_\CC(A, B) = \iso_\CC(A, B)$;

  $(d)$ $\hom_\CC(A, A) = \{\id_A\}$ for all $A \in \Ob(\CC)$;

  $(e)$ if $B \to C$ and $C \to B$ then $B \cong C$ for all $B, C \in \Ob(\CC)$.
\end{LEM}
\begin{proof}
  By the assumption, $\CC$ has a countable skeleton $\SS$ such that for every
  $S \in \Ob(\SS)$ there are only finitely many morphisms in $\SS$ whose codomain is~$S$.
  
  $(a)$ It is easy to see that $|\hom_\CC(A, B)| = |\hom_\SS(S_A, S_B)|$, where $S_A, S_B \in \Ob(\SS)$ are the unique objects in $\SS$
  isomorphic to $A$ and $B$, respectively. Since there are only finitely many morphisms in $\SS$ whose codomain is $S_B$ it follows that
  $\hom_\SS(S_A, S_B)$ is finite.

  $(b)$ Immediate from the definition.

  $(c)$ Assume that $A \cong B$ for some $A, B \in \Ob(\CC)$ and let $S \in \Ob(\SS)$ be the unique object in $\SS$ such that $A \cong S \cong B$.
  Then, as in $(a)$, we have that $|\hom_\CC(A, B)| = |\hom_\SS(S, S)|$. Since $\hom_\SS(S, S)$ is a finite left cancellable monoid
  every morphism in $\hom_\SS(S, S)$ is invertible, so $\hom_\SS(S, S) = \Aut_\SS(S)$. Note also that $\SS$ itself has the Ramsey property, whence follows that
  $\Aut_\SS(S) = \{\id_S\}$. Therefore, $\hom_\SS(S, S) = \{\id_S\}$. Now, fix isomorphisms $f_A \in \iso_\CC(A, S_A)$ and
  $f_B \in \iso_\CC(B, S_B)$ and let $h \in \hom_\CC(A, B)$ be any morphism. Then $f_B \cdot h \cdot f^{-1}_A \in \hom_\SS(S, S) = \{\id_S\}$
  whence $h = f^{-1}_B \cdot f_A \in \iso_\CC(A, B)$.

  $(d)$ Similar to $(c)$ because $A \cong A$.

  $(e)$ Immediate from $(d)$.
\end{proof}

\begin{LEM}\label{p-a-t.lem.nonthin-seq}
  Let $\CC$ be a Ramsey category of finite objects which is not thin.
  
  $(a)$ There exist $C_0, C_1, C_2, \ldots \in \Ob(\CC)$ such that $C_i \to C_{i+1}$ and $C_{i+1} \not\to C_i$ for all $i \ge 0$.

  $(b)$ There does not exist a $B \in \Ob(\CC)$ such that $A \to B$ for all $A \in \Ob(\CC)$.
\end{LEM}
\begin{proof}
  Let $\CC$ be category which is not thin and let $\SS$ be its skeleton. Without loss of generality
  it suffices to prove that the two statements hold in $\SS$.
    
  $(a)$ Note that $\SS$ is not thin, so there exist $A, B \in \Ob(\SS)$ such that $|\hom_\SS(A, B)| \ge 2$.
  Because $\CC$ is a Ramsey category there is a $C \in \Ob(\SS)$ such that $C \longrightarrow (B)^A_2$.
  Note that $B \to C$ by definition.

  \medskip

  Claim 1. $C \not\to B$.

  Proof. Suppose this is not the case. Then $C \cong B$ by Lemma~\ref{p-a-t.lem.misc}, whence follows that
  $B = C$ because $\SS$ is a skeleton.
  Let $\hom_\SS(A, B) = \{p, q, \ldots\}$ where $p \ne q$ and consider the coloring $\chi : \hom_\SS(A, B) \to 2$ defined by
  $\chi(p) = 0$ and $\chi(x) = 1$ for all $x \in \hom_\SS(A, B) \setminus \{p\}$. Then for every $w \in \hom_\SS(B, B)
  = \{\id_B\}$ we have that $|\chi(w \cdot \hom_\SS(A, B))| = 2$, which contradicts the choice of~$C$.

  \medskip

  Claim 2. $|\hom_\SS(B, C)| \ge 2$.

  Proof. Suppose this is not the case. Then $|\hom_\SS(B, C)| = 1$, say $\hom_\SS(B, C) = \{f\}$.
  Let $\hom_\SS(A, B) = \{p, q, \ldots\}$ where $p \ne q$ and consider the coloring $\chi : \hom_\SS(A, C) \to 2$ defined by
  $\chi(f \cdot p) = 0$ and $\chi(x) = 1$ for all $x \in \hom_\SS(A, C) \setminus \{f \cdot p\}$.
  Then for every $w \in \hom_\SS(B, C) = \{f\}$ we have that $|\chi(w \cdot \hom_\SS(A, B))| = 2$ because
  $\chi(f \cdot p) = 0$ and $\chi(f \cdot q) = 1$ (note that $f \cdot q \ne f \cdot p$ because morphisms in $\CC$ are mono). 
  This contradicts the choice of~$C$.

  \medskip

  Going back to the proof of the lemma, construct a sequence $C_0, C_1, C_2, \ldots$ of objects of $\SS$ as follows:
  $C_0 = B$, $C_1 = C$ and $C_{i} \longrightarrow (C_{i-1})^{C_{i-2}}_2$ for $i \ge 2$. Then
  $C_i \to C_{i+1}$, $i \ge 0$, while from Claims~1 and~2 we know that
  $|\hom_\SS(C_{i}, C_{i+1})| \ge 2$, which ensures that $C_{i+1} \not\to C_{i}$ for all $i \ge 0$.

  \medskip

  $(b)$ Suppose, to the contrary, that there is a $B \in \Ob(\SS)$ such that $A \to B$ for every $A \in \Ob(\SS)$.
  If there is an $A \in \Ob(\SS)$ such that $|\hom_\SS(A, B)| \ge 2$, take $C \in \Ob(\SS)$ such that $C \longrightarrow (B)^A_2$.
  Then by Claim~1 we have that $C \not\to B$ --- contradiction.
  
  For the other possibility, assume that $|\hom_\SS(A, B)| = 1$ for all
  $A \in \Ob(\SS)$. Take any $A_1, A_2 \in \Ob(\SS)$, let $f_1 \in \hom_\SS(A_1, B)$ and $f_2 \in \hom_\SS(A_2, B)$ be the unique
  morphisms. Let us show that $|\hom_\SS(A_1, A_2)| = 1$. Suppose that $\hom_\SS(A_1, A_2) = \{u, v, \ldots\}$ with $u \ne v$.
  then $f_2 \cdot u = f_1 = f_2 \cdot v$, whence $u = v$ after cancelling $f_2$. Contradiction.

  Therefore, $|\hom_\SS(A_1, A_2)| = 1$ for all $A_1, A_2 \in \Ob(\SS)$, so $\SS$ is thin. Contradiction.
\end{proof}

Several statements in this paper will require showing a non-reduction result (that is, showing that there is no pre-adjunction
for a pair of categories). All these non-reduction results are based on the following lemma.

\begin{LEM}\label{p-a-t.lem.card-arg}
  Let $\BB$ and $\CC$ be locally small categories such that all the morphisms in $\BB$ are mono and
  that there is a pre-adjunction $F : \Ob(\BB) \rightleftarrows \Ob(\CC) : H$. Then
  $|\hom_\CC(F(A), F(B))| \ge |\hom_\BB(A, B)|$ for all $A, B \in \Ob(\BB)$.
\end{LEM}
\begin{proof}
  Take any $A, B \in \Ob(\BB)$ and let $u = \id_{F(B)} \in \hom_\CC(F(B), F(B))$. By (PA) for every
  $f \in \hom_\BB(A, B)$ there exists an $f' \in \hom_\CC(F(A), F(B))$ such that 
  $\Phi_{A,C}(u \cdot f') = \Phi_{B,C}(u) \cdot f$. For each $f \in \hom_\BB(A, B)$ choose one such
  $f' \in \hom_\CC(F(A), F(B))$. This establishes a function $\theta : \hom_\BB(A, B) \to \hom_\CC(F(A), F(B))$ defined by $\theta(f) = f'$.
  Let us show that $\theta$ is injective.
  Take any $f_1, f_2 \in \hom_\BB(A, B)$ and let $\theta(f_1) = f'_1$ and $\theta(f_2) = f'_2$. Then
  $$
    \Phi_{A, C}(u \cdot f'_1) = \Phi_{B, C}(u) \cdot f_1 \text{ and }
    \Phi_{A, C}(u \cdot f'_2) = \Phi_{B, C}(u) \cdot f_2.
  $$
  Therefore,
  $
    f'_1 = f'_2 \Rightarrow f_1 = f_2
  $
  because $\Phi_{B, C}(u)$ is mono.
\end{proof}

\begin{THM}\label{p-a-t.thm.omega-smallest}
  Let $\omega$ be the linear order of nonnegative integers understood as a thin category, and let
  $\CC$ be a Ramsey category of finite objects. If $\CC$ is thin then $\CC \equivT 1$ or 
  $\CC \equivT \omega$. If, however, $\CC$ is not thin then $\omega \ltT \CC$.
  (Here, 1 denotes the trivial one-element category with a single identity morphism.)
\end{THM}
\begin{proof}
  Let $\CC$ be a thin Ramsey category of finite objects, and let $\SS$ be its skeleton.
  Then $\SS$ is an at most countable partial order and we have that $1 \equivT \SS \equivT \CC$ or
  $\omega \equivT \SS \equivT \CC$ (Theorem~\ref{p-a-t.thm.tukey=PA} and Lemma~\ref{p-a-t.lem.SeqpaC}).

  Assume, therefore, that $\CC$ is not thin.
  By Lemma~\ref{p-a-t.lem.nonthin-seq} there is a sequence
  $C_0, C_1, C_2, \ldots \in \Ob(\CC)$ such that $C_i \to C_{i+1}$ and $C_{i+1} \not\to C_i$ for all $i \ge 0$.
  Let us construct a pre-adjunction $F : \omega \rightleftarrows \Ob(\CC) : H$.
  Let $F(k) = C_k$ for all $k \in \omega$. To define $H$ we first let $H(C_k) = k$, $k \in \omega$, and for the
  remaining objects $X \in \Ob(\CC) \setminus \{ C_k : k \in \omega \}$ we define $H(X)$ as follows:
  \begin{itemize}
    \item if $C_i \not\to X$ for all $i \in \omega$ put $H(X) = 0$;
    \item if $C_i \to X$ for some $i \in \omega$ put $H(X) = \max\{i \in \omega : C_i \to X\}$; note that $\{i \in \omega : C_i \to X\}$
          is finite because of Lemma~\ref{p-a-t.lem.misc}~$(b)$.
  \end{itemize}
  Since $\omega$ is thin, the definition of $\Phi_{k, X} : \hom_\CC(F(k), X) \to \hom_\omega(k, H(X))$ is obvious, once we ensure that
  $\hom_\omega(k, H(X)) \ne \0$ whenever $\hom_\CC(F(k), X) \ne \0$. But this is straightforward:
  if $\hom_\CC(F(k), X) \ne \0$ then $C_k \to X$ so $k \le \max\{i \in \omega : C_i \to X\} = H(X)$.
  It is also easy to see that the condition (PA) in the definition of the pre-adjunction is satisfied, so
  $\omega \leT \CC$.

  To complete the proof we still have to show that $\CC \not\leT \omega$. Suppose, to the contrary, that $\CC \leT \omega$
  and let $F : \Ob(\CC) \rightleftarrows \omega : H$ be a pre-adjunction. Since $\CC$ is not thin there exist
  $A, B \in \Ob(\CC)$ such that $|\hom_\CC(A, B)| \ge 2$. So,
  $
    |\hom_\CC(A, B)| \ge 2 > 1 \ge \hom_\omega(F(A), F(B))
  $
  because $\omega$ is thin. Contradiction with Lemma~\ref{p-a-t.lem.card-arg}.
\end{proof}

The above result agrees with the intuition that the categories where the Ramsey property is trivial
should be the weakest Ramsey categories. However, it provides no insight into the mutual relationship of
``proper'' Ramsey statements. The following result shows that $\catR$, which encodes the Finite Ramsey Theorem,
is the weakest amongst the most significant classes of categories of finite structures.

\begin{LEM}\label{p-a-t.lem.ordered}
  Let $\KK$ be a Ramsey age of finite relational structures. Then every $\calA \in \KK$ can be expanded
  by a linear order $<^A$ so that if $f$ is an embedding $\calA \hookrightarrow \calB$ then
  $f$ is an embedding $(\calA, \Boxed{<^A}) \hookrightarrow (\calB, \Boxed{<^B})$. Consequently, the classes $\KK$ and
  $\KK' = \{(\calA, \Boxed{<^A}) : \calA \in \KK\}$ are isomorphic as categories.
\end{LEM}
\begin{proof} (Sketch)
  Let $\calF = \Flim(\KK)$. Then $\Aut(\calF)$ is extremely amenable because $\KK$ is a Ramsey age.
  Let $\LO(F)$ be the set of all the linear orders on $F$, the base set of $\calF$, endowed with the obvious topology.
  This is a compact Hausdorff space, so the natural action (by shifts) of $\Aut(\calF)$ on $\LO(F)$, being continuous,
  has a joint fixed point. In other words, there is a linear order $<^F$ on $F$ which is invariant for every
  automorphism of~$\calF$. Take any $\calA \in \KK$ and an embedding $g : \calA \hookrightarrow \calF$,
  and define $<^A$ on $A$ by pulling $<^F$ from $F$ to $A$ along $g$.
  The ultrahomogeneity of $\calF$ ensures that $<^A$ does not depend on~$g$.
  The same argument applies to show that if $f$ is an embedding $\calA \hookrightarrow \calB$ then
  $f$ is an embedding $(\calA, \Boxed{<^A}) \hookrightarrow (\calB, \Boxed{<^B})$.
\end{proof}

Let $\KK$ be an age and $\SS$ its skeleton. We say that $\KK$ is \emph{oligomorphic} if
for each $n \in \NN$ there are only finitely many structures of size $n$ in $\SS$.

\begin{THM}
  Let $\KK$ be an oligomorphic Ramsey age of finite relational structures.
  Then $\catR \leT \KK$.
\end{THM}
\begin{proof}
  Let $L$ be a relational language such that $\KK$ is an age of finite $L$-structures, and let $\SS$ be the skeleton of $\KK$.
  Lemma~\ref{p-a-t.lem.ordered} ensures that we can safely assume that there is a binary symbol $\Boxed< \in L$
  which is interpreted as a linear order in every structure in $\KK$. Let $\calF = \Flim(\KK)$.
  Let us now inductively construct a family
  $$
    A_0 \supseteq A_1 \supseteq A_2 \supseteq A_3 \supseteq \ldots
  $$
  of infinite subsets of $F$, the base set of $\calF$. To start the induction let $A_0 = F$. Assume that $A_{n-1}$ has been constructed.
  List all the $n$-elements structures in $\SS$ as $S_1, \ldots, S_{k}$ and define a $k$-coloring
  $$
    \chi : [A_{n-1}]^n \to k
  $$
  of $n$-subsets of $A_{n-1}$ so that $\chi(X) = j$ if $\restr\calF X \cong S_j$.
  By the Infinite Ramsey Theorem there is an infinite monochromatic subset $A_n \subseteq A_{n-1}$.

  Note that, by construction, for every $m, n \in \NN$ such that $m \le n$ and any pair of $m$-subsets $X, Y \in [A_n]^m$
  of $A_n$ we have that $\restr \calF X \cong \restr \calF Y$.
  Moreover,
  \begin{itemize}
    \item[(*)]
    for every $m, n \in \NN$ such that $m \le n$ and any choice of $X \in [A_m]^m$ and $Y \in [A_n]^n$
    we have that $f$ is an embedding $\restr \calF X \hookrightarrow \restr \calF Y$
    if and only if $f$ is an embedding of finite chains $(X, \Boxed{<^X}) \hookrightarrow (Y, \Boxed{<^Y})$.
  \end{itemize}
  For each $n \in \NN$ choose an $X_n \in [A_n]^n$ and
  let $\JJ$ be the category whose objects are $(X_n, \Boxed{<^{X_n}})$, $n \in \NN$,
  and whose morphisms are embeddings of finite chains. Clearly, $\JJ$ is a skeleton of $\catR$
  so $\catR \equivT \JJ$ by Lemma~\ref{p-a-t.lem.SeqpaC}. Let us show that $\JJ \leT \KK$.

  Note that for every $\calC \in \KK$ there is a unique isomorphism of chains $\eta_\calC : (C, \Boxed{<^C}) \to (X_n, \Boxed{<^{X_n}})$
  where $n = |C|$. Define $F : \Ob(\JJ) \to \Ob(\KK)$ by $F(X, \Boxed{<^X}) = \restr \calF X$ and
  $H : \Ob(\KK) \to \Ob(\JJ)$ by $H(\calC) = (X_n, \Boxed{<^{X_n}})$ where $n = |C|$, the underlying set of $\calC$.
  Finally, define
  $$
    \Phi_{X_m,\calC} : \hom_\KK(F(X_m), \calC) \to \hom_\JJ(X_m, H(\calC))
  $$
  by $\Phi_{X_m,C}(u) = \eta_\calC \cdot u$. Here we use the fact that if $u : \calB \hookrightarrow \calC$ is an embedding between
  two structures in $\KK$ then $u : (B, \Boxed{<^B}) \hookrightarrow (C, \Boxed{<^C})$ is an embedding between the corresponding chains.
  Finally, let us verify (PA). Take any $f : (X_\ell, \Boxed{<^{X_\ell}}) \hookrightarrow (X_m, \Boxed{<^{X_m}})$.
  As noted in (*) the same $f$ is an embedding $\restr \calF {X_\ell} \hookrightarrow \restr \calF {X_m}$, that is,
  $f \in \hom_\KK(F(X_\ell), F(X_m))$. Take any $\calC \in \KK$ and $u \in \hom_\KK(F(X_m), \calC)$ and note that
  $$
    \Phi_{X_\ell, \calC} (u \cdot f) = \eta_\calC \cdot u \cdot f = \Phi_{X_m,\calC}(u) \cdot f.
  $$
  This verifies (PA) and completes the proof.
\end{proof}

\section{The strongest Graham-Rothschild statement?}
\label{p-a-t.sec.GR-all-the-same}

In this section we are going to show that all the Graham-Rothschild statements are Tukey equivalent by showing that
$$
  \catD^\op \leT \GR(A, X, G) \leT \GR(\0, X, G) \leT \catD^\op,
$$
where $X = \{x_1, x_2, x_3, \ldots\}$ is a countable set of variables disjoint from~$A$.
(Recall that $\catD^\op \equivT \GR(\0, X, \{e\})$.)

In the three lemmas that follow $A$ is a finite alphabet, $X = \{x_1, x_2, x_3, \ldots\}$ is a countable set of variables
disjoint from~$A$, and $G$ is a finite group acting on $A$ from the right whose neutral element is~$e$.

\begin{LEM}
  $\GR(\0, X, \{e\}) \leT \GR(A, X, G)$.
\end{LEM}
\begin{proof}
  Let us construct a pre-adjunction
  $$
    F : \Ob(\GR(\0, X, \{e\})) \rightleftarrows \Ob(\GR(A, X, G)) : H.
  $$
  Recall that $\Ob(\GR(A, X, G)) = \NN = \{1, 2, 3, \ldots\}$ and
  take $F, H : \NN \to \NN$ to be the identity $F(n) = H(n) = n$, $n \in \NN$. For $u \in W_m^n(A, G)$
  let $\Phi_{m,n}(u) \in W_m^n(\0, \{e\})$ be the $m$-parameter $n$-letter word obtained from $u$ by replacing all
  the group elements with~$e$, and replacing all the letters from $A$ with~$x_1$. This clearly establishes a family of maps
  $
    \Phi_{m, n} : W_m^n(A, G) \to W_m^n(\0, \{e\})
  $
  which satisfies the condition (PA). (In Definition~\ref{opos.def.PA}, for $f \in W_m^n(\0, \{e\})$ take $v = f \in W_m^n(A, G)$).
\end{proof}

\begin{LEM}
  $\GR(A, X, G) \leT \GR(\0, X, G)$.
\end{LEM}
\begin{proof}
  Let us enumerate $A$ as $A = \{a_1, a_2, \ldots, a_t\}$, $t = |A|$,
  and let $Y = A \sqcup X = \{a_1, a_2, \ldots, a_t, x_1, x_2, \ldots \}$
  be a new set of variables where $a_1$ is the first and $x_1$ the $(t+1)$-th variable. Clearly,
  $\GR(\0, X, G) \cong \GR(\0, Y, G)$, so we shall prove the lemma by exhibiting a pre-adjunction
  $$
    F : \Ob(\GR(A, X, G)) \rightleftarrows \Ob(\GR(\0, Y, G)) : H.
  $$
  As a notational convenience let $\BB = \GR(A, X, G)$ and $\CC = \GR(\0, Y, G)$. Recall that
  $\Ob(\BB) = \Ob(\CC) = \NN$. Define $F : \NN \to \NN$ by $F(n) = t + n$ and $H : \NN \to \NN$ as $H(n) = n$, $n \in \NN$.
  Then define $\Phi_{m,n} : \hom_\CC(t+m, n) \to \hom_\BB(m, n)$ by $\Phi_{m,n}(u) = u$. This requires a comment:
  when $u$ is considered as an element of $\hom_\CC(t+m, n)$, this is an $n$-letter word over $t + m$ variables
  $\{a_1, \ldots, a_t, x_1, \ldots, x_m\}$; when the same word is considered as an element of
  $\hom_\BB(m, n)$, this is an $n$-letter word over $m$ variables $\{x_1, \ldots, x_m\}$ where
  $a_1, \ldots, a_t$ serve as ``constant letters'' from~$A$. Finally, to see that the condition
  (PA) in Definition~\ref{opos.def.PA} is satisfied, for $f \in \hom_\BB(n, m)$ take $v = a_1 a_2 \ldots a_t f \in \hom_\CC(t + m, t + n)$.
  Then $\Phi_{m, \ell}(u \cdot v) = \Phi_{n, \ell}(u) \cdot f$ follows from the fact that the first $t$ variables in $Y$ are
  $a_1, \ldots a_t$; since $u$ is a $(t + n)$-parameter word, the first $t+n$ variables from $Y$ are the parameters in $u$, so
  there is an initial segment of $u$ which contains all of the letters $a_1, \ldots, a_t$.
\end{proof}

\begin{LEM}\label{p-a-t.lem.0XG=>Dop}
  $\GR(\0, X, G) \leT \catD^\op$.
\end{LEM}
\begin{proof}
  Let us fix a linear ordering of $G$ in which $e$ is the least element: $G = \{e < g_2 < \ldots g_t \}$, $t = |G|$.
  As a notational convenience let $\BB = \GR(\0, X, G)$ and $\CC = \catD^\op$. The pre-adjunction
  $F : \Ob(\BB) \rightleftarrows \Ob(\CC) : H$ is constructed as follows.

  For $n \in \Ob(\BB) = \NN$ let $F(n) = \{1, 2, \ldots, n\} \times G$ be linearly ordered so that
  $(i, g) < (j, h)$ if $i < j$, or $i = j$ and $g < h$ in~$G$. For a finite chain $C \in \Ob(\CC)$ let $H(C) = |C|$, the
  number of elements of~$C$. To complete the construction we still have to define
  $$
    \Phi_{n, C} : \hom_\CC(F(n), C) \to \hom_\BB(n, H(C)),
  $$
  where $n \in \NN$ and $C = \{c_1 < c_2 < \ldots < c_\ell \} \in \Ob(\CC)$. Take any $u \in \hom_\CC(F(n), C)$.
  Then $u$ is a rigid surjection
  $$
    u : \{c_1 < c_2 < \ldots < c_\ell \} \to \{1, 2, \ldots, n\} \times G,
  $$
  so we define $\Phi_{n, C}(u)$ to be the obvious $n$-parameter $G$-decorated $\ell$-letter word
  $$
    \Phi_{n, C}(u) : \{1 < 2 < \ldots < \ell \} \to \{x_1, x_2, \ldots, x_n\} \times G
  $$
  defined by $\Phi_{n, C}(u)(i) = (x_j, g)$ if and only if $u(c_i) = (j, g)$. It is easy to see that this definition is
  correct i.e.\ that each rigid surjection gives rise to an $n$-parameter $G$-decorated $\ell$-letter word.

  \begin{figure}
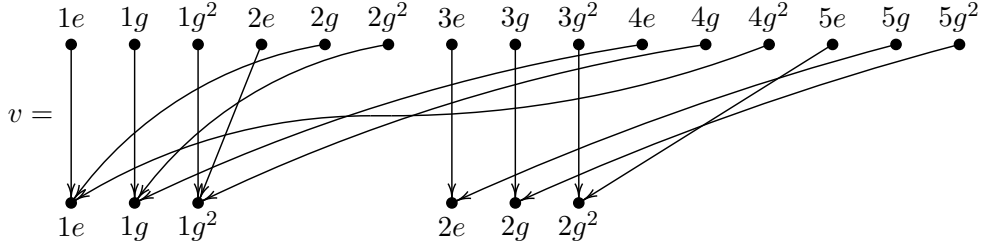

    \centering
    \input rsg.pgf
    \caption{Proof of Lemma~\ref{p-a-t.lem.0XG=>Dop}: showing by example that $v$ is a rigid surjection}
    \label{p-a-t.fig.rigid-surj}
  \end{figure}

  Let us show that this choice of $F$, $H$ and $\Phi$ satisfies (PA) from Definition~\ref{opos.def.PA}. 
  Take any $f \in \hom_\BB(m, n)$. Then $f$ is an $m$-parameter $G$-decorated $n$-letter word
  $f : \{1, 2, \ldots, n \} \to \{x_1, x_2, \ldots, x_m\} \times G$. Define
  $$
    v : \{1, 2, \ldots, n \} \times G \to \{1, 2, \ldots, m\} \times G
  $$
  so that $v(j, h) = (i, gh)$, where $f(j) = (x_i, g)$. Instead of proving formally that $v$ is a rigid
  surjection with respect to the ordering of sets of the form $F(n)$, we offer a proof by example.
  Let $G = \{e < g < g^2\}$ and $f = x_1^e x_1^{g^2} x_2^e x_1^g x_2^{g^2}$, that is,
  $$
    f = \left(\begin{array}{ccccc}
      1 & 2 & 3 & 4 & 5 \\
      x_1^e & x_1^{g^2} & x_2^e & x_1^g & x_2^{g^2}
    \end{array}\right)
  $$
  Then $v$ constructed as above is given in~Fig.~\ref{p-a-t.fig.rigid-surj}, and is clearly a rigid surjection.

  We still have to show that $\Phi_{m, C}(u \cdot v) = \Phi_{n, C}(u) \cdot f$.
  Let $u(c_k) = (j, h)$ and $f(j) = (x_i, g)$. Then $\Phi_{n,C}(u)(k) = (x_j, h)$ and $v(j, h) = (i, gh)$.
  Having in mind that $u \cdot v$ in $\CC$ is just $v \circ u$ because $\CC = \catD^\op$,
  we now have that $u \cdot v(c_k) = v \circ u(c_k) = (i, gh)$ whence
  $\Phi_{m, C}(u \cdot v)(k) = (x_i, gh) = (\Phi_{n,C}(u) \cdot f)(k)$:
    $$
      \overbrace{\ldots \ \underset{\substack{\uparrow\\k}}{x_j^h} \ \ldots}^{\Phi_{n,C}(u)}
      \quad
      \cdot
      \quad
      \overbrace{\ldots \ \underset{\substack{\uparrow\\j}}{x_i^g} \ \ldots}^{f}
      \quad
      =
      \quad
      \overbrace{\ldots \ \underset{\substack{\uparrow\\k}}{(x_i^g)^h} \ \ldots}^{\Phi_{n,C}(u) \cdot f}
    $$
  This concludes the proof.
\end{proof}

\begin{THM}
  Let $A$ be a finite alphabet, $X = \{x_1, x_2, x_3, \ldots\}$ a countable set of variables disjoint from~$A$,
  and $G$ a finite group acting on $A$ from the right. Then $\catD^\op \equivT \GR(A, X, G)$.
\end{THM}

\section{Concluding remarks}
\label{p-a-t.sec.conclusion}

In this paper we have considered only a few examples of finite Ramsey statements
(see Table~\ref{p-a-t.fig.statements-categories}), and their relative ``Ramsey strength'' can be summarized as
follows:
$$
  \omega \ltT \left\{
  \begin{array}{rcl}
    \cline{2-2}
    \catR \leT &\multicolumn{1}{|c|}{}& \geT \catH(k) \leadsto \catRel(L)\\
    \catVec(\mathbb{F}) \leT &\multicolumn{1}{|c|}{\catD^\op}& \geT \catG \equivT \catM(\{a, b\})\\
    \GR(A, X, G) \equivT &\multicolumn{1}{|c|}{}& \geT \catP \geT \catM(S)\\
    \cline{2-2}
  \end{array}
  \right.
$$
(where $a$ and $b$ are positive reals such that $a < b \le 2a$).
This immediately raises a plethora of questions but we shall confine ourselves to only a few.

\paragraph{Problem 1.}
  We have seen that $\catR \leT \catD^\op$ (this is unsurprising and easy).
  Is it true that $\catR \ltT \catD^\op$?

\paragraph{Problem 2.}
  Is it true that the Finite Dual Ramsey Theorem is the strongest result of Ramsey theory of finite structures? More precisely,
  is the following true: if $\CC$ is a Ramsey category of finite objects which is not thin then $\CC \leT \catD^\op$?

\paragraph{Problem 3.}
  We have seen that $\catR \leT \KK$ for many significant classes $\KK$ of finite relational structures with the Ramsey property.
  Is it true that the Finite Ramsey Theorem is the weakest result of Ramsey theory of finite structures? More precisely,
  is the following true: if $\CC$ is a Ramsey category of finite objects which is not thin then $\catR \leT \CC$?

\paragraph{Problem 4.}
  Well-quasi-orders usually lead to satisfactory rough classification results.
  Is $\leT$ a well-quasi-order when restricted to Ramsey categories of finite objects?

\paragraph{Problem 5.}
  Classify modulo $\equiv_T$ all Ramsey categories of finite objects. 

\bigskip

Let us conclude the paper with a brief discussion of Tukey reducibility between $\catR$ and $\catIR$, a category
whose objects are all finite chains together with all countably infinite chains of order type~$\omega$, and whose morphisms are embeddings.
The Infinite Ramsey Theorem now takes the following form:
  for every finite chain $n \in \NN$, every $k \in \NN$ and every coloring $\chi : \hom_{\catIR}(n, \omega) \to k$
  there is a $w \in \hom_{\catIR}(\omega, \omega)$ such that $|w \circ \hom_{\catIR}(n, \omega)| = 1$.

\begin{PROP}
  The categories $\catR$ and $\catIR$ are Tukey unrelated, that is,
  $\catR \not\leT \catIR$ and $\catIR \not\leT \catR$.
\end{PROP}
\begin{proof}
  Assume that $\catR \leT \catIR$. Then there is a pre-adjunction $F : \Ob(\catR) \rightleftarrows \Ob(\catIR) : H$.
  Lemma~\ref{p-a-t.lem.tukey-cofinal} ensures that $F$ is a Tukey map. But this is impossible because every family of objects in $\catIR$ is
  bounded by $\omega$, while 1, 2, 3, \ldots is an unbounded family in $\catR$.

  Assume, now, that $\catIR \leT \catR$ and let $F : \Ob(\catIR) \rightleftarrows \Ob(\catR) : H$ be a pre-adjunction.
  Then
  $$
  |\hom_{\catR}(F(1), F(\omega))| < |\hom_{\catIR}(1, \omega)|
  $$
  because $\hom_{\catIR}(1, \omega)$ is countably infinite and $\hom_{\catR}(F(1), F(\omega))$ is finite.
  Contradiction with Lemma~\ref{p-a-t.lem.card-arg}.
\end{proof}

This proposition shows that reasoning in terms of Tukey reducibility is not always compatible with the
proof-theoretic intuition. From the proof-theoretic point of view the Finite Ramsey Theorem is properly weaker than
the Infinite Ramsey Theorem: it is a standard fact that the Infinite Ramsey Theorem implies its finite version, while
a result due to Specker (see~\cite{specker}) implies that the Infinite Ramsey Theorem
does not follow from the Finite Ramsey Theorem. With respect to Tukey reducibility, however,
the two results are unrelated: one direction is not surprising, while the other comes from the fact that pre-adjunctions
generalize Tukey reducibility for preorders.

\section*{Acknowledgements}

The second author was supported by the Science Fund of the Republic of Serbia, Grant No.~7750027:
Set-theoretic, model-theoretic and Ramsey-theoretic phenomena in mathematical structures: similarity and diversity -- SMART.

% cspell:disable

\end{document}